\documentclass[10pt, reqno]{amsart}

\usepackage{amsmath,amssymb,graphicx}
\newtheorem{theorem}{Theorem}[section]
\newtheorem{lemma}[theorem]{Lemma}
\newtheorem{corollary}[theorem]{Corollary}

\newtheorem{sublemma}{}[theorem]

\newcommand{\ba}{\backslash}
\newcommand{\cl}{{\rm cl}}

\begin{document}

\title[A matroid extension result]{A matroid extension result}

\author{James Oxley}
\address{Mathematics Department, Louisiana State University, Baton Rouge, Louisiana, USA}
\email{oxley@math.lsu.edu}


\date{\today}

\begin{abstract} 
Adding elements to matroids can be fraught with difficulty. 
In the V\'amos matroid $V_8$, there are four   independent sets $X_1,X_2, X_3,$ and $X_4$ such that $(X_1 \cup X_2,X_3 \cup X_4)$ is a $3$-separation while exactly three of the local connectivities $\sqcap(X_1,X_{3})$,  $\sqcap(X_1,X_{4})$, $\sqcap(X_2,X_{3})$, and $\sqcap(X_2,X_{4})$ are one, with the fourth being zero.   As is well known, there is no extension of  $V_8$ by a non-loop element $p$ such that $X_j \cup p$ is a circuit for all $j$. This paper proves that a matroid can be extended by a fixed element in the guts of a $3$-separation provided  no V\'amos-like structure is present. 
\end{abstract}

\maketitle

\section{Introduction}
\label{intro} 



The terminology here will follow~\cite{oxrox}.   Consider the V\'amos matroid, $V_8$, the rank-$4$ paving matroid on $\{a_1,a_1',a_2,a_2',b_1,b_1',b_2,b_2'\}$ having as its non-spanning circuits $\{a_1,a_1',a_2,a_2'\}$, $\{a_1,a_1', b_1,b_1'\}$, $\{a_1,a_1',b_2,b_2'\}$, $\{a_2,a_2',b_1,b_1'\}$, and $\{b_1,b_1',b_2,b_2'\}$. Then $(\{a_1,a_1',a_2,a_2'\},\{b_1,b_1',b_2,b_2'\})$ is an exact $3$-separation $(A,B)$ of $V_8$. For each $i$ in $\{1,2\}$, we have the following local connectivity conditions: 
$\sqcap(\{a_i,a_i'\},B) = 1 =  \sqcap(\{b_i,b_i'\},A)$. Moreover, $\sqcap(\{a_1,a_1'\},\{b_1,b_1'\}) = 1$. In this situation, it is tempting to try to extend $V_8$ by a rank-one element $p$ that lies on the lines spanned by $\{a_1,a_1'\}$ and $\{b_1,b_1'\}$. But this cannot be done. The proof that this extension is impossible depends on the fact that exactly three of the four local connectivities $\sqcap(\{a_i,a_i\},\{b_j,b_j'\})$ for $\{i,j\} \subseteq \{1,2\}$ are equal to one. The purpose of this paper is to show that this condition is the sole impediment to being able to add a point $p$ to the guts of a $3$-separation $(A,B)$ of a matroid   so that each of $A \cup p$ and $B \cup p$ contains a circuit containing $p$ whose local connectivity with the other side is one.

Let $(A,B)$ be an exact $3$-separation in a matroid $M$. An {\it $A$-strand} is a minimal  subset $A'$ of $A$ for which $\sqcap(A',B) = 1$. A {\it $B$-strand} is defined symmetrically. A {\it strand} is an $A$-strand or a $B$-strand. 
The following is the main result of the paper.

\begin{theorem}
\label{bigun}
Let $(A,B)$ be an exact $3$-separation in a matroid $M$. Assume there is an $A$-strand $A_0$ and a $B$-strand $B_0$ such that $\sqcap(A_0,B_0) = 1$. Then $M$ has an extension by an element $p$ in which $A_0 \cup p$ and $B_0 \cup p$ are circuits if and only if $M$ has no $A$-strand $A_1$ and $B$-strand $B_1$ distinct from $A_0$ and $B_0$, respectively, such that exactly two of $\sqcap(A_0,B_1)$, $\sqcap(A_1,B_0)$, and $\sqcap(A_1,B_1)$ are one. Moreover, when $M$ has an extension by $p$ in which $A_0 \cup p$ and $B_0 \cup p$ are circuits, this extension is unique.
\end{theorem}

It is natural to consider performing multiple extensions of the type in the last theorem. In Section~\ref{mc}, we prove the following  result along with a natural generalization of it that allows for arbitrarily many extensions.

\begin{theorem}
\label{multi}
Let $(X,Y,Z)$ be a partition of the ground set of a matroid $M$ where $Y$ may be empty. Let $(X,Y \cup Z)$ and $(X \cup Y,Z)$  be exact $3$-separations of $M$. Assume there is an $X$-strand $X_0$ and a $(Y \cup Z)$-strand $Y_0$ of $M$ such that $\sqcap(X_0,Y_0) = 1$ and $M$ has an extension by an element $p$ so that $X_0 \cup p$ and $Y_0 \cup p$ are circuits. Assume there is an $(X\cup Y)$-strand $Y_1$ and a $Z$-strand $Z_1$ of $M$ such that $\sqcap(Y_1,Z_1) = 1$ and $M$ has an extension by an element $q$ so that $Y_1 \cup q$ and $Z_1 \cup q$ are circuits. Then $M$ has a unique extension by the elements $p$ and $q$ such that $X_0 \cup p, Y_0 \cup p, Y_1 \cup q$, and $Z_1 \cup q$  are circuits.
\end{theorem}

The next section introduces some terminology and proves some basic lemmas. The proof of Theorem~\ref{bigun} will be given in Section~\ref{proofsec}.

\section{Preliminaries}
\label{prelim}

Let $U$ and $V$ be subsets of the ground set of a matroid $M$. We say that $\{U,V\}$ is a {\it modular pair} of sets if 
$r(U) + r(V) = r(U \cup V) + r(U\cap V)$. The following  result will be useful. We omit the straightforward  proof. 

\begin{lemma}
\label{elem}
Let $\{U,V\}$ be a modular pair of sets in a matroid. If $w \in \cl(U) \cap \cl(V)$, then $w \in \cl(U \cap V)$. 
\end{lemma}

Next we note a basic property of local connectivity \cite[Lemma 2.4]{osw} that will be used frequently in what follows.

\begin{lemma}
\label{PQR}
For subsets $P$, $Q$, and $R$ of a matroid $M$,
$$\sqcap(P \cup Q,R) + \sqcap(P,Q) = \sqcap(P \cup R,Q) + \sqcap(P,R).$$
\end{lemma}

Elements $e$ and $f$ in a matroid $M$ are {\it clones} if the map that interchanges $e$ and $f$ while fixing every other element is an automorphism of $M$. Elements $g$ and $h$ of $M$ are {\it independent  clones} if they are clones and $\{g,h\}$ is independent in $M$.  An element $z$ of $M$ is {\it fixed} in $M$ if there is no extension $M'$ of $M$ by an element $z'$ such that $z$ and $z'$ are independent clones of $M'$. The next result  follows from 
Theorem 6.1 and Lemma 6.3 of Geelen, Gerards, and Whittle \cite{ggw} (see also \cite{bdb}).

\begin{lemma}
\label{exterminate}
Let $(A,B)$ be an exact $3$-separation of a matroid $M$.   Then there is a unique extension  $M'$  of $M$ by an element $x'$ such that   $x' \in \cl_{M'}(A) \cap \cl_{M'}(B)$ and $x'$ is not fixed in $M'$.
\end{lemma}

The statement of the last lemma matches that of \cite[Lemma 7.9]{gw} except that the former adds the requirement that the extension $M'$ be unique. Although uniqueness is essentially implicit in the latter, we include the proof here for completeness.

\begin{proof}[Proof of Lemma~\ref{exterminate}.] 
It is proved in \cite[Lemma 6.3]{ggw} that the set ${\mathcal F}$ of flats $F$ of $M$ such that $A-F$ is a separator of $M/F$ is a modular cut of $M$ and, moreover  \cite[(6.3.1)]{ggw}, that ${\mathcal F}$ is the unique minimal modular cut of $M$ containing $\{\cl(A),\cl(B)\}$. Corresponding to ${\mathcal F}$, there is an extension $M'$ of $M$ by the element $x'$, and $x' \in \cl_{M'}(A) \cap \cl_{M'}(B)$. Thus $(A,B\cup x')$ is an exact $3$-separation of $M'$. Hence $M'$ has an extension $M''$ by an element $x''$ for which the corresponding modular cut  ${\mathcal F}'$ consists of the flats $F'$ of $M'$ such that $\lambda_{M'/F'}(A - F') = 0$, that is, such that $A- F'$ is a separator of $M'/F'$. 

We show next \cite[Lemma 2.2.7]{bdb} that 

\begin{sublemma}
\label{ic} 
$x'$ and $x''$ are independent clones in $M''$.
\end{sublemma}

If $F' \in {\mathcal F}'$, then $r(A \cup F') + r(B \cup x' \cup F') = r(M') + r(F')$. As 
$r(A) + r(B \cup x') = r(M) + 2$, we deduce that $r(F') \ge 2$, so $\cl_{M'}(\{x'\}) \not\in {\mathcal F}'$. Hence $\{x',x''\}$ is independent. 

To show that $x'$ and $x''$ are clones, it suffices by \cite[Proposition 4.9]{govw} to show that a cyclic flat $X$ of $M''$ contains $x'$ if and only if it contains $x''$.  We omit the straightforward details of this check. 


By \ref{ic}, $x'$ is not fixed in $M'$. To establish the uniqueness of the extension $M'$, let $N'$ be an extension of $M$ by a non-fixed element $x'$ that lies in $\cl_{N'}(A) \cap \cl_{N'}(B)$. Then we can independently clone $x'$ by $x''$ and again by $x'''$ to get $N'''$. In the last matroid, let $X$ be a flat 
containing $x'$ as a non-coloop. Then $r(X \cup \{x'',x'''\}) = r(X - x')$ in $N'''$. 
As $r(A \cup \{x',x'',x'''\}) + r(B \cup \{x',x'',x'''\}) = r(M) + 2,$ it follows that $N'''/(X \cup \{x'',x'''\})$ has $A - (X-x')$ as a separator. Since $N'''/(X \cup \{x'',x'''\})$ has $x',x'',$ and $x'''$ as loops, we deduce that $N'''/(X-x')\ba  \{x',x'',x'''\})$, which equals $M/(X-x')$,  has $A - (X-x')$ as a separator. Hence $X - x' \in 
{\mathcal F}$. Because ${\mathcal F}$ is the unique minimal modular cut of $M$ containing $\{\cl(A),\cl(B)\}$, it follows that $N' = M'$, so $M'$ is unique. 
\end{proof} 

The following is an immediate consequence of the last lemma. 

\begin{corollary}
\label{exterminate2}
Let $(A,B)$ be an exact $3$-separation of a matroid $M$.   Then there is a unique extension  $M''$  of $M$ by a pair of independent clones $x$ and $y$ such that   $\{x,y\} \subseteq \cl_{M''}(A) \cap \cl_{M''}(B)$. 
\end{corollary}

In the last result, we shall say that $x$ and $y$ have been {\it freely added  to the guts line} of $(A,B)$. 

\begin{lemma}
\label{unicorn}
Let $(A,B)$ be an exact $3$-separation in a matroid $M$. Assume there is an $A$-strand $A_0$ and a $B$-strand $B_0$ such that $\sqcap(A_0,B_0) = 1$. Let $M''$ be the extension of $M$ obtained by freely adding elements $x$ and $y$ to the guts line of $(A,B)$. Then $M$ has an extension by an element $p$ such that  $A_0 \cup p$ and $B_0 \cup p$ are circuits if and only if $M''$ has an extension by   $p$ such that  $A_0 \cup p$ and $B_0 \cup p$ are circuits. 
\end{lemma}

\begin{proof}
Clearly if $M''$ has such an extension, then so does $M$. Conversely, assume that $M$ has an extension $M_p$ by $p$ in which both $A_0 \cup p$ and $B_0 \cup p$ are circuits. As $(A,B\cup p)$ is an exact $3$-separation of $M_p$, we can freely add elements $x$ and $y$ to the guts line of 
$(A,B\cup p)$ in $M_p$ to get $M''_p$. Then, in $M''_p\ba p$, the elements $x$ and $y$ are independent clones that are contained in  
$\cl_{M_p''\ba p}(A) \cap \cl_{M_p''\ba p}(B)$. By Corollary~\ref{exterminate2}, $M''_p \ba p = M''$, so $M''$ has the desired extension. 
\end{proof}

Let $(A,B)$ be an exact $3$-separation in a matroid $M$. Form a bipartite graph $G$ with vertex classes consisting of the set of $A$-strands and the set of $B$-strands. An $A$-strand $A'$ and a $B$-strand $B'$ are adjacent in $G$ if $\sqcap(A',B') = 1$. A {\it bunch} of strands is the vertex set of some component  of $G$ that has at least one edge. We call a bunch of strands {\it complete} if the associated bipartite graph is complete. Clearly a bunch of strands that contains a single $A$-strand or a single $B$-strand is complete.  We call $G$ the {\it strand graph} of $(M,A,B)$. The following  is elementary. In this  and the two subsequent lemmas, $(A,B)$ is an exact $3$-separation in a matroid $M$. 

\begin{lemma}
\label{complete} 
Let  $Z$ be a bunch of strands that contains at least two $A$-strands and at least two $B$-strands. Then $Z$ is complete if and only if, whenever 
$Z$ contains $A$-strands $A'$ and $A''$ and $B$-strands $B'$ and $B''$ such that at least three of $\sqcap(A',B'), \sqcap(A',B''), \sqcap(A'',B')$, and $\sqcap(A'',B'')$ are one, all four are one.
\end{lemma}

We omit the straightforward proof of the next result.

\begin{lemma}
\label{equiv}
Let $A'$ be an $A$-strand and $B'$ be a $B$-strand. Then $\sqcap(A',B') = 1$ if and only if $A' \cup B'$ is a circuit of $M$.
\end{lemma}


\begin{lemma}
\label{equiv2}
For $A' \subseteq A$ and $B' \subseteq B$, if    $A' \cup B'$ is a circuit of $M$ and $\sqcap(A,B') = 1$, then $B'$ is a $B$-strand.
\end{lemma}

\begin{proof}
Clearly $B'$ contains a $B$-strand $B''$. Let $A''$ be a minimal subset of $A$ such that $\sqcap(A'',B'') = 1$. Then  one easily checks that $A'' \cup B''$ is a circuit. 
Since 
$$1 = \sqcap(A',B')\le \sqcap(A' \cup A'', B') \le \sqcap(A, B') = 1,$$
 $\sqcap(A' \cup A'', B') = 1$. Thus $M|(A'\cup A'' \cup B')$ is a $2$-sum, with basepoint $q$ say, of matroids $M_1$ and $M_2$ having ground sets $B' \cup q$ and $A' \cup A''\cup q$. 
Then $M_1$ has $B' \cup q$ and $B''\cup q$ as circuits. Thus $B' = B''$. 
\end{proof}

\begin{lemma}
\label{subeq}
Let $X_1$, $X_2$, and $Y$ be sets in a matroid $M$. If $\{X_1,X_2\}$ is a modular pair, then 
$$\sqcap(X_1,Y) + \sqcap(X_2,Y) \le \sqcap(X_1 \cup X_2,Y) + \sqcap(X_1 \cap X_2,Y).$$
\end{lemma}

\begin{proof} By substitution, we see that 
\begin{multline*}
\sqcap(X_1 \cup X_2,Y) + \sqcap(X_1 \cap X_2,Y) - \sqcap(X_1,Y) - \sqcap(X_2,Y)  \\
= [r(X_1 \cup Y) + r(X_2 \cup Y) - r(X_1\cup X_2 \cup Y) - r((X_1\cap X_2) \cup Y)] \\
- [r(X_1) + r(X_2) - r(X_1 \cup X_2) - r(X_1 \cap X_2)].
\end{multline*}
The result follows since $\{X_1,X_2\}$ is modular and $r$ is submodular.
\end{proof}

\begin{corollary}
\label{subeq2}
In a matroid $M$, suppose $X_1$, $X_2$, and $Y$ are sets  and $\{X_1,X_2\}$ is a modular pair. 
If $\sqcap(X_1,Y) =  \sqcap(X_1 \cup X_2,Y)$, then $\sqcap(X_2,Y) = \sqcap(X_1 \cup X_2,Y)$.
\end{corollary}

\begin{proof}
As $\sqcap(X_1 \cup X_2,Y) \ge \sqcap(X_2,Y)$, this is immediate from the last lemma.  
\end{proof}

\section{The Proof of the Main Result}
\label{proofsec}

The purpose of this section is to prove the main theorem.

\begin{proof}[Proof of Theorem~\ref{bigun}.]
Suppose first that $M$ has an extension   by an element $p$ such that $A_0 \cup p$ and $B_0 \cup p$ are circuits. Assume also that  $M$ has an  $A$-strand $A_1$ and $B$-strand $B_1$ distinct from $A_0$ and $B_0$, respectively, such that exactly two of $\sqcap(A_0,B_1)$, $\sqcap(A_1,B_0)$, and $\sqcap(A_1,B_1)$ are one. If both $A_1 \cup p$ and $B_1 \cup p$ are circuits, then it follows by circuit elimination that each of  
$A_0 \cup B_0, A_0 \cup B_1, A_1 \cup B_0$, and  $A_1 \cup B_0$ contains a circuit each of which must meet both $A$ and $B$. Since all of $A_0, A_1, B_0,$ and $B_1$ are strands, it follows that all of $A_0 \cup B_0, A_0 \cup B_1, A_1 \cup B_0$, and  $A_1 \cup B_0$ are circuits. Hence, by Lemma~\ref{equiv}, we obtain the contradiction that $\sqcap(A_i,B_j) = 1$ for all $i$ and $j$ in $\{0,1\}$.

Now assume that $M$ has no  $A$-strand $A_1$ and $B$-strand $B_1$ distinct from $A_0$ and $B_0$, respectively, such that exactly two of $\sqcap(A_0,B_1)$, $\sqcap(A_1,B_0)$, and $\sqcap(A_1,B_1)$ are one. By Corollary~\ref{exterminate2}, there is an extension $M''$ of $M$ by a pair of independent clones $x$ and $y$ where $\{x,y\} \subseteq \cl_{M''}(A) \cap \cl_{M''}(B)$.   Note that a flat of $M''$ that contains a circuit containing $x$ or $y$ must contain  the line $L$ of $M''$ that is spanned by $\{x,y\}$. 

To prove that $M$ has the desired extension, we shall show that $M''$ has an extension $M'$ by the element $p$ so that $A_0 \cup p$ and $B_0 \cup p$ are circuits. Consider the component of the strand graph of $(M,A,B)$ that contains $A_0$ and $B_0$. Call a strand that labels a vertex in this component {\it special}. For each subset $X$ of $E(M'')$, we define $r(X) = r_{M''}(X)$ and 
\begin{equation*}
r(X \cup p) = 
\begin{cases}
r_{M''}(X) & \text{if $X$ contains a special strand or $\cl_{M''}(X) \supseteq L$};\\
r_{M''}(X) + 1 & \text{otherwise.}
\end{cases}
\end{equation*}
Let ${\mathcal F}$ be the set of subsets $F$ of $E(M'')$ such that $r(F \cup p) = r_{M''}(F).$

We shall complete the proof  that $M''$ has the desired extension $M'$ by the element $p$ by showing that $r$ is a matroid rank function. Assume the contrary. Then it is straightforward to see that $r$ is not submodular. Thus there are subsets $X$ and $Y$ of $E(M'') \cup \{p\}$ such that 
\begin{equation}
\label{eq1}
r(X) + r(Y) < r(X \cup Y) + r(X \cap Y).
\end{equation}
Clearly 
\begin{equation}
\label{eq2}
r(X-p) + r(Y-p) \ge r((X-p) \cup (Y-p)) + r((X-p) \cap (Y-p)).
\end{equation}


For some $\alpha$ and $\beta$ in $\{0,1\}$, we have 
$r(X- p) = r(X) - \alpha$ and $r(Y- p) = r(Y) - \beta$.
Then 
\begin{align*}
r(X- p) + r(Y-p) & = r(X) - \alpha + r(Y) - \beta\\ 
& \le [r(X \cup Y) - \alpha - \beta] + [r(X \cap Y) - 1]\\
& \le r((X \cup Y) - p) + r((X \cap Y) - p),
\end{align*}
where the last step is immediate if $\alpha + \beta > 0$ and also holds if $\alpha + \beta = 0$. 
By (\ref{eq2}), equality must hold throughout the last chain of inequalities. Thus $r((X \cap Y) - p) = r(X \cap Y)-1$ 
so $p \in X \cap Y$. Moreover, $\alpha = 0$ or $\beta = 0$, so $\alpha = \beta = 0$. Thus $X - p, Y- p$, and $(X\cup Y) - p$ are in ${\mathcal F}$, but $(X\cap Y) - p \notin {\mathcal F}$.
Writing $X'$ and $Y'$ for $X - p$ and $Y-p$, respectively, we see that
\begin{equation}
\label{eqmod}
r(X') + r(Y') = r(X' \cup Y') + r(X' \cap Y').
\end{equation}

We now choose a pair $\{X',Y'\}$ of subsets  of $E(M'')$ with $X',Y', X' \cup Y' \in {\mathcal F}$ and $X' \cap Y' \not\in {\mathcal F}$ so that $|X' \cup Y'|$ is a minimum. Next we show the following where  the closure operator in $M''$ has been abbreviated to $\cl$.

\setcounter{theorem}{1}

\begin{sublemma}
\label{notboth}
At least one of $\cl(X')$ and $\cl(Y')$ does not contain $L$.
\end{sublemma}

Assume both $\cl(X')$ and $\cl(Y')$   contain $L$. Then $\sqcap(X',L) = 2 = \sqcap(Y',L)$. Thus $\sqcap(X'\cup Y',L) = 2$. Hence, by Lemma~\ref{subeq}, $\sqcap(X'\cap Y',L) = 2$, so $X'\cap Y' \in {\mathcal F}$; a contradiction. Therefore \ref{notboth} holds.

In the proof of the next assertion, it will be useful to recall that, in $M''$, the elements $x$ and $y$ both lie on the line $L$ and neither is fixed. 
\begin{sublemma}
\label{p40}
Suppose $A' \subseteq A$ and $B' \subseteq B$. If $\sqcap(A',B') = 2$, then $\sqcap(A'\cup B',L) = 2$.
\end{sublemma}

Suppose $\sqcap(A'\cup B',L) \neq 2$. Then $x \notin \cl(A' \cup B')$. Thus  
$r(A' \cup B' \cup x) = r(A' \cup B') + 1$, so $r(A'  \cup x) = r(A') + 1$ and $r(B' \cup x) = r(B') + 1$. Now $x \in \cl(A) \cap \cl(B)$. Thus 
\begin{align*}
2 = \sqcap(A',B') & \le \sqcap(A' \cup x,B'\cup x)\\
& \le \sqcap(A \cup x,B\cup x)\\
& = \sqcap(A,B) =2.
\end{align*}
Thus $\sqcap(A' \cup x,B'\cup x) = 2$, so  
\begin{align*}
2 & = r(A' \cup x) + r(B' \cup x) - r(A' \cup B' \cup x)\\
&  = r(A') + 1 + r(B') + 1 - r(A' \cup B') - 1\\
& = \sqcap(A',B') + 1\\
& =2 + 1.
\end{align*}
This contradiction completes the proof of \ref{p40}.

Recall that $\sqcap(A,L) = 2 =  \sqcap(B,L).$

\begin{sublemma}
\label{p23}
For $Z\subseteq A$, 
$$\sqcap(Z,L) = \sqcap(Z,B) = \sqcap(Z,B \cup L).$$
\end{sublemma}

We have 
\begin{align*}
\sqcap (Z,B) & = r(Z) + r(B) - r(Z \cup B)\\
& = r(Z) + r(B \cup L) - r(Z \cup B \cup L)\\
& = \sqcap(Z,B \cup L).
\end{align*}
By Lemma~\ref{PQR}, 
\begin{equation}
\label{eqn9}
\sqcap(Z,B \cup L) + \sqcap(B,L) = \sqcap(Z\cup L,B) + \sqcap(Z,L).
\end{equation} 
Now $\sqcap(B,L) = 2$, so $2 \le \sqcap(Z \cup L,B) \le \sqcap(A \cup L,B) = \sqcap(A,B) = 2$. Hence, by (\ref{eqn9}),  
$\sqcap(Z,B \cup L)= \sqcap(Z,L)$ so \ref{p23} holds.

\begin{sublemma}
\label{p16.8}
For $Z\subseteq E(M'')$, 
$$\sqcap(Z,L) + \sqcap(Z\cap A, Z \cap B) = \sqcap(Z\cap A,L) + \sqcap(Z\cap B,L).$$
\end{sublemma}

By Lemma~\ref{PQR}, 
$$\sqcap((Z\cap A) \cup (Z \cap B),L) + \sqcap(Z\cap A, Z \cap B) = \sqcap((Z\cap A) \cup L,Z \cap B) + \sqcap(Z\cap A,L).$$
By \ref{p23} and symmetry, 
\begin{align*}
\sqcap(Z \cap B,A) = \sqcap(Z \cap B,L) & \le \sqcap(Z \cap B,(Z\cap A) \cup L)\\
& \le  \sqcap(Z \cap B,A \cup L) = \sqcap(Z\cap B,A).
\end{align*}
Thus $\sqcap(Z \cap B,L) = \sqcap(Z \cap B,(Z\cap A) \cup L).$ Therefore \ref{p16.8} holds.

\begin{sublemma}
\label{strand}
For $A' \subseteq A$ and $B' \subseteq B$, if $\sqcap(A' \cup B',L) = 1 = \sqcap(B',L)$ and $A'$ contains a special strand, then $B'$ also contains a special strand.
\end{sublemma}

Let $A_1$ be a special strand contained in $A'$. As $\sqcap(B',L) = 1$, it follows by \ref{p23} that $\sqcap(B',A) = 1$, so $B'$ contains a $B$-strand,  $B_1$ say. Now, by \ref{p16.8}, 
$$\sqcap(A_1 \cup B_1,L) + \sqcap(A_1,B_1) = \sqcap(A_1,L) + \sqcap(B_1,L) = 2.$$
As $1 = \sqcap(B_1,L) \le \sqcap(A_1 \cup B_1,L) \le \sqcap(A'  \cup B',L) = 1$, we deduce that $\sqcap(A_1 \cup B_1,L) = 1$, so $\sqcap(A_1,B_1) = 1$. Hence $B_1$ is a special $B$-strand so \ref{strand} holds.  

\begin{sublemma}
\label{prestranded}
Suppose $\cl(X') \not \supseteq L$. If $X' \cap A$ contains a strand, then $\sqcap(X'\cap A,L) = 1 = \sqcap(X',L)$. If, in addition,   $X'\cap B$ contains a strand, then  $\sqcap(X'\cap B,L) = 1 = \sqcap(X'\cap A,X'\cap B)$. 
\end{sublemma}

As $\sqcap(X',L) < 2$, we know that $\sqcap(X'\cap A,L) \le 1$. Since $X' \cap A$ contains a strand, $\sqcap(X' \cap A, B) \ge 1$, so $\sqcap(X' \cap A, L) \ge 1$. Hence $\sqcap(X' \cap A, L) =  1$ and $\sqcap(X' , L) = 1$. Now suppose $X'\cap B$ also contains a strand. Then $\sqcap(X' \cap B, L) =  1$. Moreover, by \ref{p16.8}, $\sqcap(X'\cap A,X'\cap B) = 1$. Hence \ref{prestranded} holds.

Recall that $X',Y' \in {\mathcal F}$.

\begin{sublemma}
\label{stranded}
Suppose 
$\cl(X') \not \supseteq L$. Then one of the following occurs.
\begin{itemize}
\item[(i)] 
Both $X'\cap A$ and $X' \cap B$ contain special strands, 
$\sqcap(X'\cap A,L) = 1 = \sqcap(X'\cap B,L)$ and $\sqcap(X'\cap A,X'\cap B) = 1$; or 
\item[(ii)] $X'\cap A$ contains a special strand, $X'\cap B$ does not contain a strand, $\sqcap(X'\cap A,L) = 1$ and  $\sqcap(X'\cap B,L) = 0 = \sqcap(X'\cap A,X'\cap B)$; or 
\item[(iii)] $X'\cap B$ contains a special strand, $X'\cap A$ does not contain a strand, $\sqcap(X'\cap B,L) = 1$ and  $\sqcap(X'\cap A,L) = 0 = \sqcap(X'\cap A,X'\cap B)$. 
\end{itemize}
\end{sublemma}

Since  $X' \in {\mathcal F}$ but $\cl(X') \not \supseteq L$, at least one of $X' \cap A$ and $X'\cap B$ contains a special strand. Assume $X'\cap A$ does. Then, by \ref{prestranded}, $\sqcap(X' \cap A, L) =  1$ and, if $X'\cap B$ also contains a strand, then   $\sqcap(X' \cap B, L) =  1 = \sqcap(X'\cap A,X'\cap B)$. As $\sqcap(X,L) = 1$, it follows  by \ref{strand} that  $X' \cap B$ contains a special strand, so  (i) holds. 

Now suppose $X'\cap B$ does not contain a strand. Then, by \ref{p16.8}, $\sqcap(X'\cap A,X'\cap B) = \sqcap(X'\cap B,L)$. If this quantity is $0$, then (ii) holds and \ref{stranded} is proved. Thus we may assume that $\sqcap(X'\cap A,X'\cap B) = 1= \sqcap(X'\cap B,L)$. Then, by \ref{p23}, 
$\sqcap(X'\cap B,A) = 1$, so $X'\cap B$ contains a strand; a contradiction. Hence 
\ref{stranded} holds.

Now, for $Z$ in $\{A,B\}$, let 
$$\gamma(Z) = r(X' \cap Z) + r(Y' \cap Z) - r((X' \cup Y') \cap Z) - r(X' \cap Y'\cap Z).$$
By submodularity, $\gamma(Z) \ge 0.$

\begin{sublemma}
\label{sums}
\begin{multline*}0= \gamma(A) + \gamma(B)\\
+ [\sqcap((X' \cup Y') \cap A,(X' \cup Y') \cap B) - \sqcap(Y' \cap A,Y' \cap B)]\\
- 
 [\sqcap(X' \cap A,X' \cap B) - \sqcap(X'\cap Y' \cap A,X' \cap Y' \cap B)].
 \end{multline*}
 \end{sublemma}
 
 To see this, observe, by (\ref{eqmod}) that 
 \begin{align*}
 0 & = r(X') + r(Y') - r(X' \cup Y') - r(X' \cap Y')\\
 & = [r(X'\cap A) + r(X' \cap B) - \sqcap(X' \cap A,X' \cap B)]\\
  & \hspace*{0.1in} +  [r(Y'\cap A) + r(Y' \cap B) - \sqcap(Y' \cap A,Y' \cap B)]\\
  & \hspace*{0.2in}- [r((X'\cup Y')\cap A) + r((X'\cup Y') \cap B) - \sqcap((X'\cup Y') \cap A,(X'\cup Y') \cap B)]\\
  & \hspace*{0.3in} -  [r(X' \cap Y'\cap A) + r(X' \cap Y' \cap B) - \sqcap(X' \cap Y' \cap A, X' \cap Y' \cap B)]\\
 & = \gamma(A) + \gamma(B) + [\sqcap((X' \cup Y') \cap A,(X' \cup Y') \cap B) - \sqcap(Y' \cap A,Y' \cap B)] \\
 & \hspace*{0.3in} - 
 [\sqcap(X' \cap A,X' \cap B) - \sqcap(X'\cap Y' \cap A,X' \cap Y' \cap B)].
 \end{align*}
 
By \ref{notboth},  at most one of $\cl(X')$ and $\cl(Y')$ contains $L$. We shall now assume that 

\begin{sublemma}
\label{assum}
$\cl(Y')$ contains $L$ but $\cl(X')$ does not. 
\end{sublemma}
By symmetry, we may also assume that (i) or (ii) of \ref{stranded} holds. Thus 
$$\sqcap(X'\cap A,L) = 1$$ and $X' \cap A$ contains a special strand.  Of course, $A_0$ and $B_0$ are special strands.  

\begin{sublemma}
\label{y'a3}
If  $\{X' \cap A, Y' \cap A\}$ is a modular pair   and 
$\sqcap(Y'\cap A,L) = 1$, then 
\begin{itemize}
\item[(i)] 
$\sqcap((X'\cup Y') \cap A,L) = 2$; and 
\item[(ii)] $\sqcap(((X'\cup Y')   \cap A) \cup L, (X'\cup Y')   \cap B) =  \sqcap((X'\cup Y') \cap A,(X'\cup Y') \cap B)$. 
\end{itemize}
\end{sublemma}

Assume $\sqcap((X'\cup Y') \cap A,L) < 2$. Then, since $\sqcap(Y'\cap A, L) = 1$, we see that 
$\sqcap((X'\cup Y')\cap A, L) = 1$. Now, by \ref{p16.8}, 
$$\sqcap(((X' \cup Y') \cap A) \cup B_0, L) + \sqcap((X' \cup Y') \cap A,B_0) =  \sqcap((X' \cup Y') \cap A,L) + \sqcap(B_0,L) = 2.$$
Since $X' \cap A$ contains a special strand, $\sqcap((X' \cup Y') \cap A,B_0) \ge 1$. We deduce, since 
$\sqcap(((X' \cup Y') \cap A) \cup B_0, L) \ge 1$, that $\sqcap((X' \cup Y') \cap A,B_0) = 1 = \sqcap(((X' \cup Y') \cap A) \cup B_0, L)$. 

By \ref{p16.8} again,  
$$\sqcap((Y' \cap A) \cup B_0,L) + \sqcap(Y' \cap A, B_0) = \sqcap(B_0,L) + \sqcap(Y' \cap A,L) = 2.$$
As each term on the left-hand side is at most one, we deduce that each equals one. Thus, by \ref{strand}, $Y'\cap A$ contains a special strand, so $Y' \cap A \in {\mathcal F}$. 

As $\sqcap(Y',L) = 2$, it follows  that $Y'\cap B \neq \emptyset$. Thus $|(X' \cup Y') \cap A| < |X' \cup Y'|$ and we get a contradiction to the choice of $\{X',Y'\}$ since $X' \cap Y'\cap A$, and hence $X' \cap Y'$,   is in ${\mathcal F}$. We conclude that $\sqcap((X'\cup Y') \cap A,L) =2$, so \ref{y'a3}(i) holds. 

 By Lemma~\ref{PQR}, 
\begin{align*}
\sqcap(((X' \cup Y') \cap A) \cup L, (X' \cup Y') \cap B) & + \sqcap ((X' \cup Y') \cap A,L)\\ 
& = \sqcap(X' \cup Y', L)\\
 & \hspace*{0.5in}+ \sqcap((X' \cup Y') \cap A,(X' \cup Y') \cap B).
\end{align*}
By (i), $\sqcap ((X' \cup Y') \cap A,L) = 2$, so 
$$\sqcap(((X' \cup Y') \cap A) \cup L, (X' \cup Y') \cap B)   =  \sqcap((X' \cup Y') \cap A,(X' \cup Y') \cap B),$$
that is,  \ref{y'a3}(ii) holds.

\begin{sublemma}
\label{y'a0}
If  $\{X' \cap A, Y' \cap A\}$ is a modular pair, then 
$\sqcap(Y' \cap A, L) = 1$. 
\end{sublemma}

Assume that $\sqcap(Y' \cap A, L) = 2$. Then all of 
$X' \cap A, Y'\cap A$, and $(X' \cup Y') \cap A$ are in ${\mathcal F}$. Hence, by the choice of $\{X',Y'\}$, we see that 
$(X'\cup Y') \cap B = \emptyset$ otherwise $X' \cap Y' \cap A$, and hence $X' \cap Y'$, is in ${\mathcal F}$; a contradiction.  

Now $X' \cap A = X'$ and $Y' \cap A = Y'$. As $X'$ contains a special strand, $\sqcap(X',B_0) = 1$.  By Lemma~\ref{PQR}, 
$$\sqcap (Y' \cup B_0,L) + \sqcap(Y',B_0) = \sqcap (B_0 \cup L,Y') + \sqcap(B_0,L).$$ 
Since $\sqcap(Y',L) = 2$, it follows that $\sqcap(Y',B_0) = 1$. 

As $1 = \sqcap(X',B_0) \le \sqcap(X' \cup Y',B_0) \le \sqcap(A,B_0) = 1,$ Corollary~\ref{subeq2} implies that $\sqcap(X' \cap Y',B_0) = \sqcap(Y',B_0) = 1$. 
Thus $1 \le \sqcap(X' \cap Y',L) \le \sqcap(X',L) = 1$. By \ref{p16.8}, 
$\sqcap((X' \cap Y') \cup B_0,L) + \sqcap(X' \cap Y',B_0) = \sqcap(X' \cap Y',L) + \sqcap(B_0,L),$ so $\sqcap((X' \cap Y') \cup B_0,L) = 1$. Thus, by \ref{strand}, $X' \cap Y'$ contains a special strand; a contradiction. We conclude that $\sqcap(Y' \cap A, L) < 2.$

Now assume that  $\sqcap(Y'\cap A,L) = 0$. By \ref{p16.8},
$$\sqcap(Y',L) + \sqcap(Y' \cap A,Y' \cap B) = \sqcap(Y' \cap A,L) + \sqcap(Y'\cap B,L).$$
Hence $\sqcap(Y'\cap B,L) = 2$ and $\sqcap(Y' \cap A,Y' \cap B) = 0$.  By  interchanging the terms 
$\sqcap(Y' \cap A, Y'\cap B)$ and $\sqcap(X' \cap A, X'\cap B)$ in \ref{sums}, we see that $\{X' \cap B, Y' \cap B\}$ is a modular pair. As $X' \cap A$ contains a special strand, it is non-empty. Therefore  $X' \cap B \not\in {\mathcal F}$ otherwise   $X' \cap Y' \cap B$, and hence $X' \cap Y'$, is in ${\mathcal F}$; a contradiction. Thus $X' \cap B$ does not contain a special strand. Hence (ii) of \ref{stranded} holds, so $\sqcap(X' \cap A, X'\cap B) = 0$.
Thus, by \ref{sums}, 
$\sqcap((X' \cup Y') \cap A, (X' \cup Y') \cap B) = 0.$ By \ref{p16.8} again, 
\begin{align*}
2 & = 2 +  \sqcap((X' \cup Y') \cap A, (X' \cup Y') \cap B)\\
& =  \sqcap((X' \cup Y') \cap A, L) + \sqcap((X' \cup Y') \cap B,L)\\
& \ge \sqcap(X' \cap A, L) + \sqcap(Y' \cap B,L)\\
& = 1 + 2.
\end{align*}
This contradiction implies that $\sqcap(Y' \cap A, L) = 1$, that is, \ref{y'a0} holds. 


\begin{sublemma}
\label{notii}
Case (i) of \ref{stranded} must hold.
\end{sublemma}
 
Assume instead that \ref{stranded}(ii) holds. Then $\sqcap(X'\cap A,X'\cap B) = 0$ so the second square-bracketed term in \ref{sums} is $0$.  Since the other square-bracketed term is non-negative as are each of $\gamma(A)$ and $\gamma(B)$, we deduce that $\gamma(A) = 0$. Thus $\{X'\cap A,Y' \cap A\}$ is a modular pair. Hence, by \ref{y'a0}, $\sqcap(Y' \cap A, L) = 1$.


Now, by \ref{y'a3}(ii), $\sqcap(((X'\cup Y')   \cap A) \cup L, (X'\cup Y')   \cap B) =  \sqcap((X'\cup Y') \cap A,(X'\cup Y') \cap B)$. Thus, by \ref{sums}, 
as $\sqcap(X' \cap A,X' \cap B) = 0$, 
\begin{equation}
\label{eqcross}
\sqcap(((X'\cup Y')   \cap A) \cup L, (X'\cup Y')   \cap B) =  \sqcap(Y' \cap A,Y' \cap B).
\end{equation}

By \ref{p16.8}, $\sqcap(Y',L) + \sqcap(Y'\cap A,Y'\cap B) = \sqcap(Y'\cap A,L) + \sqcap(Y'\cap B, L),$ so 
$2 +  \sqcap(Y'\cap A,Y'\cap B) = 1 + \sqcap(Y'\cap B, L),$ that is, 
$ \sqcap(Y'\cap A,Y'\cap B) =  \sqcap(Y'\cap B, L) -1.$ Therefore, by (\ref{eqcross}), 
$$\sqcap(L, Y'\cap B) \le \sqcap(((X'\cup Y')   \cap A) \cup L, (X'\cup Y')   \cap B) = \sqcap(Y' \cap B, L) - 1;$$
a contradiction. We conclude that \ref{notii} holds.

We now know that both $X'\cap A$ and $X' \cap B$ contain special strands and $\sqcap(X' \cap A, X' \cap B) = 1$. We have symmetry between $A$ and $B$ so, by \ref{sums}, we may assume that $\{X' \cap A, Y' \cap A\}$ is a modular pair. Thus, by \ref{y'a0}, 
$\sqcap(Y'\cap A,L) = 1$.

Next we observe, by \ref{p16.8}, that either 
\begin{itemize}
\item[(a)] $\sqcap(Y'\cap B,L) = 2$ and $\sqcap(Y'\cap A,Y' \cap B) = 1$; or 
\item[(b)] $\sqcap(Y'\cap B,L) = 1$ and $\sqcap(Y'\cap A,Y' \cap B) = 0$.
\end{itemize}

Suppose (a) holds. Then $Y' \cap B$ and $X' \cap B$ are in  ${\mathcal F}$. Hence $\{X'\cap B, Y' \cap B\}$ is not a modular pair otherwise we obtain the contradiction that $X' \cap Y' \cap B$, and hence $X' \cap Y'$, is in ${\mathcal F}$. Thus $\gamma(B) \ge 1$. By \ref{y'a3}(ii), 
$\sqcap((X' \cup Y')\cap A, (X' \cup Y')\cap B) \ge \sqcap(L,Y' \cap B) = 2$. Since $\sqcap(Y'\cap A,Y' \cap B) = 1 = \sqcap(X'\cap A,X' \cap B)$, we obtain a contradiction by \ref{sums}. We deduce that (b) holds. 

As  $\sqcap(Y'\cap A,Y' \cap B) = 0$, using \ref{sums} again with the terms  $\sqcap(Y' \cap A,Y' \cap B)$ and  $\sqcap(X' \cap A,X' \cap B)$ interchanged, we see that, as $\sqcap(X' \cap A, X' \cap B) = 1$, that  $$\sqcap((X' \cup Y')\cap A,(X' \cup Y')\cap B) = 1$$ and $\{X'\cap B, Y'\cap B\}$ is a modular pair. 
 We may now apply \ref{y'a3} interchanging $A$ and $B$ to get that 
 $\sqcap((X' \cup Y') \cap B, L) = 2$. By \ref{p16.8}, 
 $$\sqcap(X' \cup Y',L) + \sqcap((X' \cup Y') \cap A,(X' \cup Y') \cap B) = \sqcap((X' \cup Y') \cap A,L) + \sqcap((X' \cup Y') \cap B,L).$$
It follows that $\sqcap((X' \cup Y') \cap A,(X' \cup Y') \cap B) = 2$. This contradiction completes the argument  that $r$ is a matroid rank function when \ref{assum} holds.
 
 It remains to treat the case when neither $X'$ nor $Y'$ spans $L$. Then each of $X'$ and $Y'$ contains a special strand. By \ref{p40}, 
 $\sqcap(X'\cap A,X'\cap B) \le 1$ and  $\sqcap(Y'\cap A,Y'\cap B) \le 1$. If either 
 $\sqcap(X'\cap A,X'\cap B)$ or $\sqcap(Y'\cap A,Y'\cap B)$ is $0$, then, by \ref{sums} and symmetry, both 
 $\{X'\cap A,Y'\cap A\}$ and $\{X'\cap B,Y'\cap B\}$ are modular pairs.
 
 \begin{sublemma}
\label{allones}
$\sqcap(X'\cap A,X'\cap B) = 1 = \sqcap(Y'\cap A,Y'\cap B)$.
\end{sublemma}
 
Assume $\sqcap(X'\cap A,X'\cap B) = 0$. We know $X' \cap A$ or $X' \cap B$ contains a special strand. Then, by \ref{prestranded}, exactly one of $X' \cap A$ and $X'\cap B$, say $X' \cap A$, contains a strand.
Assume $Y' \cap A$ contains a special strand. Then $(X' \cup Y') \cap B = \emptyset$ otherwise 
$|(X' \cup Y') \cap A| < |X' \cup Y'|$ and we have a contradiction to the choice of $\{X',Y'\}$. We can now write $X'$ and $Y'$ for $X' \cap A$ and $Y' \cap A$. 

As $X' \cup Y'$ contains a special strand, $\sqcap(X' \cup Y',B_0) = 1$. Then $M|(X' \cup Y' \cup B_0)$ can be written as a 2-sum with basepoint $q$ of matroids $M_1$ and $M_2$ with ground sets $X' \cup Y' \cup q$ and $B_0 \cup q$. As each of $X'$ and $Y'$ contains a special strand, it follows that, in $M_1$, the element $q$ is in the closures of both $X'$ and $Y'$. Hence, by Lemma~\ref{elem}, $q \in \cl_{M_1}(X' \cap Y')$. Thus $M|(X' \cup Y' \cup B_0)$ has a circuit of the form $B_0 \cup A_2$ where $A_2 \subseteq X' \cap Y'$. As $1 \le \sqcap(A_2,B) \le \sqcap (X' \cap Y',B) = \sqcap(X'\cap Y', L) \le \sqcap(X',L) = 1$, it follows by Lemma~\ref{equiv2} that $A_2$ is a special strand contained in $X' \cap Y'$; 
a contradiction.

We now know that $Y' \cap A$ does not contain a special strand. Thus $Y' \cap B$ does. 
Assume $\sqcap(Y' \cap A,L) = 1$. Then, as $Y'\cap B$ contains a special strand, \ref{strand} implies that $Y'\cap A$ also contains a special strand; a contradiction. 
Thus $\sqcap(Y' \cap A, L) = 0$. As $\sqcap(Y' \cap B, L) = 1$, \ref{p16.8} implies that  
  $\sqcap(Y' \cap A,Y' \cap B) = 0$. Recall that we also know that $\sqcap(X' \cap A, L) = 1$ and $\sqcap(X' \cap B, L) = 0 = \sqcap(X' \cap A,X' \cap B)$. By \ref{sums}, 
$\sqcap((X' \cup Y') \cap A, (X' \cup Y') \cap B) = 0$. But $X' \cap A$ and $Y' \cap B$ both contain special strands, so $\sqcap(X' \cap A, Y' \cap B) \ge 1$; a contradiction. We conclude that $\sqcap(X'\cap A,X'\cap B) > 0$. By \ref{p40}, $\sqcap(X'\cap A,X'\cap B) < 2$. Hence $\sqcap(X'\cap A,X'\cap B) = 1$ and, by symmetry, \ref{allones} follows.

Since $\sqcap(X',L) = 1$, we deduce by \ref{p16.8} that $\sqcap(X' \cap A,L) = 1 = \sqcap(X' \cap B,L)$. Now $X'\cap A$ or $X' \cap B$, say $X' \cap A$,  contains a special strand, say $A_1$. By \ref{p16.8} again, 
$\sqcap(A_1 \cup (X' \cap B),L) + \sqcap(A_1,X' \cap B) = \sqcap(A_1,L) + \sqcap(X'\cap B,L)$, so $\sqcap(A_1,X'\cap B) = 1$. Thus $M$ has a circuit $A_1 \cup B_1$ for some $B_1 \subseteq X' \cap B$. Now, by \ref{p23},
$$1 = \sqcap(A_1,B_1) \le \sqcap(A,B_1)= \sqcap(L,B_1) \le \sqcap(L,X' \cap B) = 1,$$ 
so $\sqcap(A,B_1)$ = 1. Thus, by Lemma~\ref{equiv2}, $B_1$ is a $B$-strand. Since $\sqcap(A_1,B_1) = 1$, $B_1$ is a special strand. We conclude that $X'\cap A$ and $X' \cap B$ both contain special strands. By symmetry, both $Y' \cap A$ and $Y'\cap B$ also contain special strands.

By \ref{sums}, $\{X' \cap A, Y'\cap A\}$ or $\{X' \cap B, Y'\cap B\}$, say the former, is a modular pair.  But $|(X' \cup Y') \cap A| < |X' \cup Y'|$. As each of $X'\cap A$, $Y'\cap A$,  and $(X' \cup Y') \cap A$ is in ${\mathcal F}$, so is  $X' \cap Y' \cap A$. Thus  $X' \cap Y'  \in {\mathcal F}$. This contradiction completes the proof that $r$ is a matroid rank function, so $M$ has the desired extension.

To finish the proof of the theorem, it remains to show that $M'$ is  the unique extension of $M$ by $p$ such that $A_0 \cup p$ and $B_0 \cup p$ are circuits.  By the first part of the theorem and Lemma~\ref{complete}, the bunch of special strands of $M$ is complete. Let $M_p$ be an arbitrary extension of $M$ by $p$ in which $A_0 \cup p$ and $B_0 \cup p$ are circuits. Then $(A, B \cup p)$ is an exact $3$-separation of $M_p$ so it has an extension $M''_p$ by freely adding $x$ and $y$ to the guts line of $(A,B \cup p)$. Then, by Corollary~\ref{exterminate2}, $M''_p \ba p = M' \ba p = M''$. To complete the proof of the uniqueness of $M_p$, we shall show that $M''_p = M'$. 

For a subset $X$ of $E(M) \cup \{x,y\}$, we know that, in $M'$, we have  $r(X \cup p) = r(X)$ if and only if $X$ contains a special strand or $X$ spans $\{x,y\}$ in $M''$. Let $r_2$ be the rank function of $M''_p$.  We show next that 

\begin{sublemma}
\label{sss}
$r_2(X \cup p) = r_2(X)$ if $\cl_{M''}(X) \supseteq L$ or if $X$ contains a special strand. 
\end{sublemma}

Suppose $\cl_{M''}(X) \supseteq L$. Then $p \in \cl_{M_p''}(A) \cap \cl_{M_p''}(B) = \cl_{M_p''}(L)$, so $r_2(X \cup p) = r_2(X)$. On the other hand, suppose $X$ contains a special strand.  If this strand is $A_0$ or $B_0$, then certainly  $r_2(X \cup p) = r_2(X)$. Hence, by symmetry, we may assume the special strand contained in $X$ is a $B$-strand, $B_1$. We may also assume that $B_1 \cup p$ does not contain a circuit of $M_p''$, otherwise 
$r_2(X \cup p) = r_2(X)$.   As the bunch of special strands of $M$ is complete,  $\sqcap(A_0, B_1) = 1$, so $A_0 \cup B_1$ is a circuit of $M$. As $M_p''$ has $A_0 \cup B_1$ and $A_0 \cup p$ as circuits, it has a circuit $C$ that contains $p$ and avoids some element $a$ of $A_0$. As $B_1 \cup p$ does not contain a circuit, $C$ must contain some element, say $a'$, of $A_0$ Now, by strong circuit elimination, $M_p''$ has a circuit $D$ that contains $a'$ and is contained in $(C \cup B_0) - p$. As $D \cap A \subsetneqq A_0$ and $A_0$ is a strand, it follows that $\sqcap(D,L) = 0$. But this is a contradiction since $D$ is a circuit meeting both $A$ and $B$. We conclude that \ref{sss} holds. 

To complete the proof of the uniqueness of $M_p$, we assume that $Z$ is a minimal subset of $E(M'')$ such that $r_2(Z \cup p) = r_2(Z)$ but $\cl_{M''}(Z) \not \supseteq L$ and  $Z$ does not contain a special strand.   The contradiction we obtain will imply that $M''_p = M'$. By minimality, $Z \cup p$ must be a circuit of $M''_p$.

\begin{sublemma}
\label{zztop} 
$Z \cap \{x,y\} = \emptyset.$
\end{sublemma}

Clearly $\{x,y\} \not\subseteq Z.$  Suppose $|Z \cap \{x,y\}| = 1$. Then we may assume that $x \in Z$. As $y$ is a clone of $x$ in $M''_p$, we deduce that $(Z-x) \cup y \cup p$ is a circuit of $M''_p$. Thus $Z \cup y$ contains a circuit of $M''$ containing $y$. Hence $\cl_{M''}(Z) \supseteq L$; a contradiction. Thus \ref{zztop} holds.

Assume first that $Z \subseteq A$. As $Z \cup p$ and $B_0 \cup p$ are circuits of $M''_p$, we deduce that $Z \cup B_0$ contains a circuit $Z' \cup B_0'$ of $M''_p$, which must also be a circuit of $M$. Now $1 \le \sqcap(A,B'_0) \le \sqcap(A,B_0) = 1$. Thus, as $B'_0$ is a $B$-strand, $B'_0 = B_0$. 
Suppose $\sqcap(Z',B) = 1$. Then $Z'$ is an $A$-strand of $M$ by Lemma~\ref{equiv2}. As $Z' \cup B_0$ is a circuit, $Z'$ is a special strand; a contradiction. We deduce that $\sqcap(Z',B) = 2$, so, by \ref{p23}, $\sqcap(Z',L) = 2$. Hence $\cl_{M''}(Z') \supseteq L$; a contradiction. 

We may now assume that $Z = A_Z \cup B_Z$ where $A_Z = A \cap Z$ and $B_Z = B \cap Z$ and both $A_Z$ and $B_Z$ are non-empty. We also know that neither $A_Z$ nor $B_Z$ contains a special strand. Both  $B_0 \cup p$ and $Z \cup p$ are circuits of $M_2$. Take $a$ in $A_Z$. Then $M''_p$ has a circuit that contains $a$ and avoids $p$. This circuit is a circuit of $M''$. Thus $\sqcap(A_Z,B) \ge 1$. As $\sqcap(A_Z,L) \neq 2$, we deduce that $\sqcap(A_Z,B) = 1$. Hence $A_Z$ contains an $A$-strand $A'_Z$ of $M$. Likewise, $B_Z$ contains a  $B$-strand $B'_Z$ of $M$. 
As $A'_Z \cup B'_Z$ is a proper subset of the circuit $Z \cup p$ of $M''_p$, it follows that $A'_Z \cup B'_Z$ is independent in $M$. Thus 
  $\sqcap(A'_Z,B'_Z) = 0$. Also, $\sqcap(A'_Z \cup B'_Z, L) \neq 2$ as $\sqcap(Z,L) \neq  2$. Now, by Lemma~\ref{PQR}, 
$$\sqcap(A'_Z \cup B'_Z, L) + \sqcap(A'_Z,B'_Z) = \sqcap(A'_Z \cup L, B'_Z) + \sqcap(A'_Z,L).$$
Since the right-hand side is at least two but the left-hand side is at most one, we have a contradiction. We conclude that the extension $M_p$ of $M$ is unique.
\end{proof}

\section{Multiple Extensions}
\label{mc}

In Theorem~\ref{bigun},  we gave conditions for the existence of a certain  extension of a matroid by a fixed element in the guts of an exact $3$-separation. We begin this section by proving Theorem~\ref{multi}, which allows us to do two   extensions of the type in Theorem~\ref{bigun}. Using Theorem~\ref{multi} will enable us to establish the following more general result. 

\begin{theorem}
\label{multi2}
Let $T$ be an $n$-vertex tree whose vertices are labelled by non-empty disjoint sets $X_1,X_2,\dots,X_n$. Let $E = X_1 \cup X_2\cup \dots \cup X_n$ and let $M$ be a matroid with ground set $E$. Suppose that, for every edge $e$ of $T$, the induced partition $(Y_e,Z_e)$ of $E$ is an exact $3$-separation of $M$ and there are $Y_e$- and $Z_e$-strands $Y_{e1},Y_{e2},\dots,Y_{em_e}$ and $Z_{e1},Z_{e2},\dots,Z_{em_e}$ such that, for each $k$ in $\{1,2,\dots,m_e\}$, the local connectivity $\sqcap(Y_{ek},Z_{ek}) = 1$ and there is an extension of $M$ by an element $p_{ek}$ in which $Y_{ek} \cup p_{ek}$ and $Z_{ek} \cup p_{ek}$ are circuits. Then $M$ can be   extended by $\bigcup_{e \in E(T)} \{p_{e1},p_{e2},\dots, p_{em_e}\}$   to produce a matroid $M'$ in which $Y_{ek} \cup p_{ek}$ and $Z_{ek} \cup p_{ek}$ are circuits for all $e$ in $E(T)$ and all $k$ in $\{1,2,\dots,m_e\}$. Moreover, the matroid $M'$ is unique. 
\end{theorem}

\begin{proof}[Proof of Theorem~\ref{multi}.] 
By Corollary~\ref{exterminate2}, we can freely add elements $x$ and $y$ to the guts line of $(X,Y\cup Z)$ to get a unique extension $M''$ of $M$. 

\begin{sublemma} 
\label{ext1}
$M''$ has an extension by $p$ in which $X_0 \cup p$ and $Y_0 \cup p$ are circuits. Moreover, $M''$ has an extension by $q$ in which  $Y_1 \cup q$ and $Z_1 \cup q$ are circuits.
\end{sublemma}

By Lemma~\ref{unicorn}, $M''$ has an extension   by $p$ in which $X_0 \cup p$ and $Y_0 \cup p$ are circuits. Now $M$ has an extension $M_q$ by $q$ such that $Y_1 \cup q$ and $Z_1 \cup q$ are circuits. As $M_q$ has $(X,Y\cup Z \cup q)$ as an exact $3$-separation, by   Corollary~\ref{exterminate2}, we can freely add elements $x$ and $y$ to the guts line of $(X,Y\cup Z \cup q)$ to get an extension $M_q''$ of $M_q$. Then, in  $M_q'' \ba q$, the elements $x$ and $y$ are independent clones on the guts line of  $(X,Y\cup Z)$. The uniqueness of $M''$ implies that 
$M_q'' \ba q = M''$. We conclude that \ref{ext1} holds. 

It will be convenient to work with the elements $x$ and $y$. Thus, in the argument that follows, we  replace $M''$ by $M$. This means that we   assume that $\{x,y\} \subseteq E(M)$. Indeed, we  assume that $\{x,y\} \subseteq Y$ noting that $x$ and $y$ are independent clones in $M$, and $\{x,y\} \subseteq \cl_M(X) \cap \cl_M((Y - \{x,y\}) \cup Z)$. 

Let $M_p$ be the extension of $M$ by the element $p$ such that $X_0 \cup p$ and $Y_0 \cup p$ are circuits. We want to show that $M_p$ has an extension by $q$ in which $Y_1 \cup q$ and $Z_1 \cup q$ are circuits. If such an extension exists, it is unique. Observe that $Y_1$ is an $(X \cup Y \cup p)$-strand of $M_p$, that $Z_1$ is a $Z$-strand of $M_p$, and that $\sqcap(Y_1,Z_1) = 1$. We assume that $M_p$ does not have the desired extension by $q$. Then, by Theorem~\ref{bigun}, $M_p$ has an $(X\cup Y \cup p)$-stand $Y_2$ and a $Z$-strand $Z_2$ such that exactly two of 
$\sqcap(Y_1,Z_2)$, $\sqcap(Y_2,Z_1)$, and $\sqcap(Y_2,Z_2)$ are one. Clearly $Z_2$ is a $Z$-strand of $M$. Because $M$ has an extension by $q$ in which $Y_1 \cup q$ and $Z_1 \cup q$ are circuits, it follows that $p \in Y_2$ otherwise $\sqcap(Y_1,Z_2)$, $\sqcap(Y_2,Z_1)$, and $\sqcap(Y_2,Z_2)$ give a violation of Theorem~\ref{bigun}.

Let $L$ be the line of $M_p$ that is spanned by $\{x,y\}$. Then $p \in L$. Now, for $k$ in $\{1,2\}$, Lemma~\ref{equiv} implies that $\sqcap(Y_2,Z_k) = 1$  if and only if $Y_2 \cup Z_k$  is a circuit of $M_p$.  Hence 

\begin{sublemma}
\label{subs1}
$M_p$ has  $Y_2 \cup Z_i$  as a circuit for some $i$ in $\{1,2\}$.
\end{sublemma}

We divide the rest of the argument into two cases based on whether or not  $Y_2 \cap X$   is empty. 
Suppose first that $Y_2 \cap X  = \emptyset$. 

\begin{sublemma}
\label{newcirc}
For $k$ in $\{1,2\}$, the set $Y_2 \cup Z_k$ is a circuit of $M_p$ if and only if $X_0 \cup (Y_2 - p) \cup Z_k$ is a circuit of $M$. 
\end{sublemma}

As $\sqcap(X_0,Y\cup Z \cup p) = 1$ and $M_p$ has $X_0 \cup p$ as a circuit,   $M_p\ba (X - X_0)$ is the parallel connection with basepoint $p$ of $M_p|(X_0 \cup p)$ and $M_p \ba X$. Thus   \ref{newcirc} holds.

As $Y_2 \cup Z_i$ is a circuit of $M_p$, we see that $X_0 \cup (Y_2 - p) \cup  Z_i$ is a circuit of $M$. Next we show that 

\begin{sublemma}
\label{newstrand}
$X_0 \cup (Y_2 - p)$ is an $(X \cup Y)$-strand of $M$; and  $\sqcap(X_0 \cup (Y_2 - p), Z_k) =  \sqcap(Y_2, Z_k)$ for each $k$ in $\{1,2\}$. 
\end{sublemma}

By Lemma~\ref{PQR}, we have  
$$\sqcap(X_0 \cup (Y_2 - p), Z) + \sqcap(X_0, Y_2 - p) = \sqcap(X_0, (Y_2 - p)  \cup  Z) + \sqcap(Y_2 - p, Z).$$
As $M_p$ has $Y_2$ as an $(X \cup Y \cup p)$-strand, $\sqcap(Y_2 - p, Z) = 0$.  It follows, since $\sqcap(X_0,L) = 1$,  that 
$\sqcap(X_0 \cup (Y_2 - p), Z) \le 1$. As $M$ has $X_0 \cup (Y_2 - p) \cup  Z_i$ as a circuit, $\sqcap(X_0 \cup (Y_2 - p), Z) = 1$. Thus  Lemma~\ref{equiv2} implies that $X_0 \cup (Y_2 - p)$ is an $(X \cup Y)$-strand of $M$. The second part of \ref{newstrand} follows immediately by combining Lemma~\ref{equiv} with \ref{newcirc} since each of $\sqcap(X_0 \cup (Y_2 - p), Z_k)$ and  $\sqcap(Y_2, Z_k)$ is in $\{0,1\}$. 

Now $\sqcap(Y_1,Z_1) = 1$ and exactly two of $\sqcap(Y_1,Z_2)$, $\sqcap(Y_2,Z_1)$, and $\sqcap(Y_2,Z_2)$ are one. Thus, by \ref{newstrand}, exactly two of $\sqcap(Y_1,Z_2)$, $\sqcap(X_0 \cup (Y_2 - p),Z_1)$, and $\sqcap(X_0 \cup (Y_2 - p),Z_2)$ are one. Hence  the $(X \cup Y)$-strand $X_0 \cup (Y_2 - p)$ and the $Z$-strand $Z_2$ of $M$     contradict   Theorem~\ref{bigun}. 
 This completes the argument when $Y_2 \cap X  = \emptyset$.

We may now assume that $Y_2 \cap X  \neq \emptyset$. We show first that 

\begin{sublemma}
\label{subs2}
$\sqcap(Y_2,L) = 2$ and $\sqcap(Y_2\cap X,L) = 1$.
\end{sublemma}

Since $p \in Y_2 \cap L$, we see that $\sqcap(Y_2,L) \ge 1$.  From the circuit $Y_2 \cup Z_i$, we deduce that 
$\sqcap(Y_2 \cap X, (Y_2 - X) \cup Z_i) = 1$, so $\sqcap(Y_2 \cap X, Y \cup Z) \ge 1$. Hence $\sqcap(Y_2 \cap X, L) \ge 1$.  As $p \in Y_2 - X$, it follows that $Y_2 \cap X$ does not span $p$ otherwise $(Y_2 \cap X) \cup p$ contains a circuit that is properly contained in $Y_2 \cup Z_i$.

As $Y_2$ is independent,  $\sqcap(Y_2 \cap X, Y_2 - X) = 0$. Thus, by Lemma~\ref{PQR},
\begin{eqnarray*}
2 \ge  \sqcap(Y_2,L) & = &\sqcap((Y_2 \cap X) \cup  (Y_2 - X), L) + \sqcap(Y_2 \cap X, Y_2 - X)\\
& = &\sqcap(Y_2 \cap X, (Y_2 - X) \cup L) + \sqcap(Y_2 - X, L)\\
& \ge &\sqcap(Y_2 \cap X,  L) + \sqcap(Y_2 - X, L)\\
& \ge & 1 + 1,
\end{eqnarray*}
where the last inequality follows because $p \in Y_2  - X$. We deduce that \ref{subs2} holds. 

Since $Y_2 \cup Z_i$ is a circuit, it now follows from Lemma~\ref{equiv2} that 

\begin{sublemma}
\label{subs3}
$Y_2 \cap X$ is an $X$-strand of $M_p$. 
\end{sublemma}

Because $\sqcap(Y_2 \cap X, L) = 1$, we see that $M_p \ba (X - Y_2)$ has 
$(Y_2 \cap X, E(M_p) - X)$ as a $2$-separation. As $Y_2 \cup Z_i$ is a circuit of $M_p$,  the elements of $Y_2 \cap X$ are in series in $M_p \ba (X - Y_2)$. Pick $u$ in $Y_2 \cap X$ and contract the elements of $(Y_2 \cap X) - u$ from $M_p\ba (X - Y_2)$. In this matroid $M'_p$, the set $\{u,p\} \cup (Y_2 - (X \cup p)) \cup Z_i$ is a circuit. Restricting $M'_p$ to the union of this circuit and $\{x,y\}$, we see that the resulting matroid is the $2$-sum with basepoint $w$ of a $5$-point line $\{u,p,x,y,w\}$ and a circuit $w \cup (Y_2 - (X \cup p)) \cup Z_i$. We deduce that 

\begin{sublemma}
\label{subs7}
$(Y_2 - (X \cup p)) \cup Z_i \cup \{x,y\}$ and $(Y_2 \cap X)   \cup \{x,y\}$ are circuits of $M$.
\end{sublemma}

We show next that 

\begin{sublemma}
\label{subs5}
$\sqcap((Y_2  - (X \cup p)) \cup \{x,y\}, Z) = 1$.
\end{sublemma}

As $\{x,y\}$ spans $p$, we see that 
$\sqcap((Y_2  - (X \cup p)) \cup \{x,y\}, Z) =  \sqcap((Y_2  -  X) \cup \{x,y\}, Z).$ By Lemma~\ref{PQR}, 
$$\sqcap((Y_2  -  X) \cup \{x,y\}, Z) + \sqcap(Y_2  -  X, \{x,y\})  = \sqcap(\{x,y\}, (Y_2  -  X) \cup   Z) +  \sqcap(Y_2  -  X , Z).$$
As $Y_2$ is an $(X \cup Y \cup p)$-strand meeting $X$, the right-hand side is at most $2$.   Now $\sqcap(Y_2  -  X, \{x,y\}) \ge 1$, so 
 $\sqcap((Y_2  -  X) \cup \{x,y\}, Z) \le 1$. The result follows by \ref{subs7}.

By Lemma~\ref{equiv2}, \ref{subs7},  and \ref{subs5}, we deduce that 

\begin{sublemma}
\label{subs6}
$(Y_2  - (X \cup p)) \cup \{x,y\}$ is an $(X \cup Y)$-strand of $M$.
\end{sublemma}

We now apply Theorem~\ref{bigun}. Suppose that $\sqcap(Y_1,Z_2) = 0$. Then, for each $k$ in $\{1,2\}$, we have $\sqcap(Y_2,Z_k) = 1$, so  $Y_2 \cup Z_k$ is a circuit of $M_p$. Hence, replacing $Z_i$ by $Z_k$ in the proof of \ref{subs7} gives that  $(Y_2  - (X \cup p)) \cup \{x,y\} \cup Z_k$ is a  circuit of $M$. Thus 
$\sqcap(Y_2  - (X \cup p)) \cup \{x,y\}, Z_k) = 1$  and the $(X \cup Y)$-strand $(Y_2  - (X \cup p)) \cup \{x,y\}$ and the $Z$-strand $Z_2$ yield a  contradiction to Theorem~\ref{bigun} in $M$. 

We may now assume that 
 $\sqcap(Y_1,Z_2) = 1$. As $\sqcap(Y_2,Z_i) = 1$, we have    $\sqcap(Y_2,Z_j) = 0$ where $\{i,j\} = \{1,2\}$. Now $(Y_2 - (X \cup p)) \cup Z_i \cup \{x,y\}$ is a circuit of $M$, so, by Theorem~\ref{bigun}, $\sqcap(Y_2 - (X \cup p)) \cup \{x,y\}, Z_j) = 1$. Hence, by Lemma~\ref{equiv}, 
 $M$ has $(Y_2 - (X \cup p))  \cup \{x,y\} \cup Z_j$ as a circuit. Thus 
 
\begin{sublemma}
\label{subs4}
$r((Y_2 - (X \cup p)) \cup Z_j \cup \{x,y\})  = |Y_2 - X| + |Z_j|.$
\end{sublemma}

Now 
\begin{eqnarray*}
r(Y_2 \cup Z_j) & = & r(Y_2 \cup Z_j\cup \{x,y\}) \text{~~as $\sqcap(Y_2,L) = 2$;}\\
 & =  &  r((Y_2-p) \cup Z_j\cup \{x,y\})\\
 & =  &  r((Y_2 \cap X) \cup (Y_2-(X \cup p) \cup Z_j \cup \{x,y\} )\\
 & \le &  r(Y_2 \cap X) + r((Y_2-(X \cup p)  \cup Z_j \cup \{x,y\}) - 1 \text{~~as $\sqcap(Y_2 \cap X,L) = 1$;}\\
 & \le &  |Y_2 \cap X| + (|Y_2 - X| + |Z_j|) - 1  \text{~~by \ref{subs4};}\\
 & = &  |Y_2| + |Z_j| - 1.
\end{eqnarray*}
Thus $\sqcap(Y_2,Z_j) \ge 1$; a contradiction.
\end{proof}

\begin{proof}[Proof of Theorem~\ref{multi2}.] 
Take some edge $f$ of $T$. We can extend $M$ by $p_{f1}$ to get a matroid $M_1$ in which $Y_{f1} \cup p_{f1}$ and $Z_{f1} \cup p_{f1}$ are circuits. By Theorem~\ref{multi}, for all pairs $(e,k)$ other than $(f,1)$  for which $e \in E(T)$ and $k \in \{1,2,\dots, m_e\}$,   we can extend $M_1$ by $p_{ek}$ so that 
$Y_{ek} \cup p_{ek}$ and $Z_{ek} \cup p_{ek}$ are circuits. We now repeat this process using $M_1$ in place of $M$. Continuing in this way, it is clear that we will obtain the required result. At each stage of the process, the matroid we obtain is unique, so the matroid obtained at the conclusion of the process is unique.
\end{proof}

\section*{Acknowledgements.} The author thanks Jim Geelen for very helpful discussions concerning this paper and, in particular, for privately conjecturing Theorems~\ref{bigun} and \ref{multi}.

\end{document}